\documentclass{hha}
\usepackage{tikz-cd}
\usepackage{amssymb}
\usepackage{enumerate}
\usepackage{booktabs}
\usepackage[unicode]{hyperref}
\usepackage[noadjust]{cite}

\newcommand{\Gr}{\mathtt{Gr}}
\DeclareMathOperator{\vertex}{Vt}
\DeclareMathOperator{\edge}{Edge}
\DeclareMathOperator{\edgei}{Edge_i}
\DeclareMathOperator{\op}{op}
\DeclareMathOperator{\Iso}{Iso}
\DeclareMathOperator{\Ob}{Ob}
\DeclareMathOperator{\Id}{Id}
\newcommand{\Set}{\mathtt{Set}}
\DeclareMathOperator{\Sc}{Sc}
\newcommand{\wheel}{\circlearrowright}
\newcommand{\Gammaw}{\Gamma_\wheel}
\newcommand{\gwheelc}{\Gr^\wheel_{\text{c}}}
\newcommand{\gupcset}{\Set^{\Gamma^{\op}}}
\newcommand{\Gammawimage}{\Gammaw^\blacksquare}
\newcommand{\Gammawout}{\Gammaw^{\operatorname{out}}}
\newcommand{\wproperad}{\mathtt{Properad}^{\wheel}}
\DeclareMathOperator{\inp}{in}
\DeclareMathOperator{\out}{out}
\newcommand{\sK}{\mathsf{K}}
\newcommand{\sP}{\mathsf{P}}
\newcommand{\gwheelcset}{\Set^{\Gammaw^{\op}}}
\newcommand{\sSet}{\mathtt{sSet}}
\newcommand{\gwheelcsset}{\sSet^{\Gammaw^{\op}}}
\newcommand{\gupcsset}{\sSet^{\Gamma^{\op}}}
\DeclareMathOperator{\map}{map}
\newcommand{\oldB}{\mathcal A}
\newcommand{\oldC}{\mathcal B}
\DeclareMathOperator{\sections}{Sec}

\title{On factorizations of graphical maps}

\author{Philip Hackney}
\email{philip@phck.net} 
\address{Department of Mathematics,
         Macquarie University,
         NSW 2109,
         Australia}

\author{Marcy Robertson}
\email{marcy.robertson@unimelb.edu.au}
\address{School of Mathematics and Statistics, 
         The University of Melbourne,
         Victoria 3010,
         Australia}

\author{Donald Yau}
\email{dyau@math.osu.edu}
\address{Department of Mathematics,
         The Ohio State University at Newark,
         Newark, OH,
         USA}

\date{January 16, 2018}

\begin{document}

\begin{abstract}
We study the categories governing infinity (wheeled) properads.
The graphical category, which was already known to be generalized Reedy, is in fact an Eilenberg-Zilber category. 
A minor alteration to the definition of the \emph{wheeled} graphical category allows us to show that it is a generalized Reedy category.
Finally, we present model structures for Segal properads and Segal wheeled properads.
\end{abstract}

\keywords{Reedy category, graphical set, dendroidal set, Quillen model structure, Eilenberg-Zilber category, properad, wheeled properad.}

\classification{55P48, 18G30, 18D50, 18G55, 18G30, 55U35.}

\received{Month Day, Year}   
\revised{Month Day, Year}    
\published{Month Day, Year}  
\submitted{Bill Murray}      
\volumeyear{2014} 
\volumenumber{16} 
\issuenumber{2}   
\startpage{1}     
\articlenumber{1} 

\maketitle

\section{Introduction}

This paper is one part of a larger project laying the foundations of `up-to-homotopy' properads.
Properads \cite{vallette} are devices like operads and props which control certain (bi)algebraic structures, such as Lie bialgebras and Frobenius algebras.
They are strictly more general than operads (which can model algebraic and coalgebraic structures, but not simultaneously) and strictly less general than props (which are capable of modeling structures like Hopf algebras).

In analogy with the situation for categories and operads (see e.g. \cite{juliesurvey}), there should be a variety of models for such up-to-homotopy properads.
In \cite{hry15}, we constructed a Quillen model structure on the category of simplicially-enriched properads.
This model structure generalized similar model structures in the category \cite{bergner} and operad \cite{cm-simpop} settings. 
On the other hand, in \cite{hry} we gave a definition for $\infty$-properads as graphical sets (i.e., presheaves on a category $\Gamma$ of graphs) satisfying an inner horn filling condition. This generalized similar definitions for quasi-categories\footnote{Also known as $\infty$-categories or restricted Kan complexes.} \cite{bv,htt,joyaltheoryqcat} and dendroidal inner Kan complexes \cite{mw}.
In future work, we will connect the two concepts by means of a homotopy coherent nerve functor which goes from the category of simplicially-enriched properads to the category of graphical sets.
The analogue of the homotopy coherent nerve in the categorical and operadic settings is a right Quillen equivalence (\cite[2.2.5.1]{htt} \& \cite[8.15]{cm-simpop}).

When exploring the properties of the homotopy coherent nerve functor and of the category 
of graphical sets, we took advantage of certain properties also possessed by many of the (generalized) Reedy categories encountered in practice, such as the simplicial category $\Delta$. 
More specifically, there is a collection of related definitions\footnote{We expect that the indicated implications are the only ones that hold, though we lack an example showing that there are categories which are EZ in the sense of \cite{bm} that are not Eilenberg-Zilber in the sense of \cite{minimalfib} and examples of elegant Reedy categories which are not EZ-Reedy. It should be noted that the coequalizer category $\mathcal C = \mathcal C^+ = \{ 0 \rightrightarrows 1 \rightarrow 2 \}$ with precisely four non-identity maps is EZ-Reedy in the sense of \cite{bergnerrezk} but not EZ in the sense of \cite{bm}. Thus none of the other definitions implies EZ \cite{bm}. EZ-categories which happen to be strict Reedy categories are elegant by \cite[3.4]{bergnerrezk}.}
of categories of \emph{Eilenberg-Zilber type} (see the left two columns of Figure \ref{EZ cats fig}); one common feature of all of these is that maps in the inverse category are split epimorphisms.
These definitions are tied in to having a good theory of skeletal filtrations \cite[\S 6 -- 7]{bm}, a frequent coincidence of the injective and Reedy model structures \cite[3.10]{bergnerrezk}, a robust theory of minimal fibrations \cite[5.3]{minimalfib}, and so on.
In any case, in \cite{cm-ho,cm-ds,cm-simpop}, the fact that the dendroidal category $\Omega$ is a skeletal category\footnote{This condition is slightly weaker than the Eilenberg-Zilber categories of \cite{minimalfib}, as it is only required that $\mathcal{R}^-$ be \emph{contained in} the set of split epimorphisms.} in the sense of \cite[D\'efinition 8.1.1]{MR2294028} is used frequently, and in a seemingly essential way.

\begin{figure}
\begin{tikzcd}
\text{EZ-Reedy \cite{bergnerrezk}} \rar{\text{ \cite[4.1]{bergnerrezk} }} \dar & \text{Elegant \cite{bergnerrezk}} \rar & \text{strict Reedy} \dar \\
\text{Eilenberg-Zilber \cite{minimalfib}} \arrow[rr, bend right=10] \dar & \text{EZ-category \cite{bm}} \rar & \text{dualizable generalized Reedy \cite{bm}} \\
\text{Skeletal \cite{MR2294028}} 
\end{tikzcd} 
\caption{Conditions of Eilenberg-Zilber type}\label{EZ cats fig}
\end{figure}
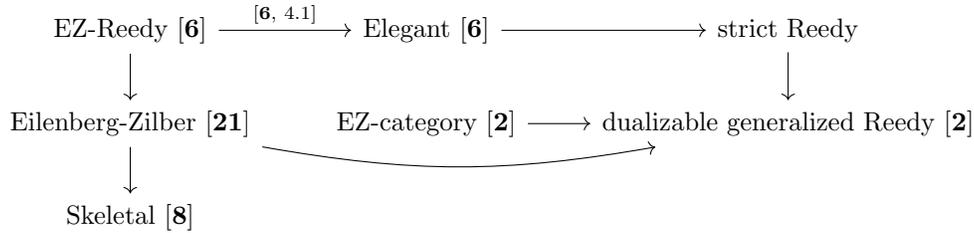

\begin{theorem}\label{gamma ez}
	The graphical category $\Gamma$ is an Eilenberg-Zilber category in the sense of \cite{minimalfib} and an EZ-category in the sense of \cite{bm}.
\end{theorem}

We have previously demonstrated that $\Gamma$ is a dualizable generalized Reedy category \cite[6.70]{hry}, and this is an extension of this fact.
Since it seems that the two definitions are incomparable, we have shown both in Theorem \ref{EZ2 theorem} and Corollary \ref{EZ1 COROLLARY}. 

Let us now consider the case of wheeled properads \cite{MR2483835}. 
The third author has produced a model structure on simplicially-enriched wheeled properads \cite{yau_wheel} and our group has introduced a category $\Gammaw$ in \cite{hry}, along with a definition for $\infty$-wheeled properads (as presheaves satisfying an inner horn filling condition).
It is natural to ask to which extent properties of $\Gamma$ and properads can be promoted to $\Gammaw$ and wheeled properads. 
One apparent asymmetry in this story is the fact that $\Gammaw$ does not admit the structure of a generalized Reedy category.
In the first part of this paper, we show that this lack is illusory, and that a minor tweak to the definition gives a generalized Reedy structure (Definition \ref{def:greedy}).

\begin{theorem}\label{gammaw reedy}
With the modification from Section \ref{modification section}, the wheeled graphical category $\Gammaw$ is a dualizable generalized Reedy category.
\end{theorem}

We further explain why this modification to the definition of the wheeled graphical category does not have any negative impacts on the existing theory.
With Theorem \ref{gammaw reedy} in hand, one could ask about Eilenberg-Zilber type structures on $\Gammaw$. As we show in Theorem \ref{gammaw not ez}, such structures do not exist -- not every codegeneracy admits a section. We interpret this fact as a significant obstruction to comparison theorems (\`a la \cite{cm-ho,cm-ds,cm-simpop}) for various potential models for up-to-homotopy wheeled properads. 

In the final section, we indicate that certain categories of (simplicial) presheaves of $\Gamma$ and $\Gammaw$ may be endowed with Quillen model structures which identify objects satisfying a Segal condition. 
The fibrant objects in these categories may be considered as a weakened version of monochrome (wheeled) properads.
This section rests on machinery from \cite{hco}, the fact that $\Gamma$ and $\Gammaw$ are generalized Reedy categories, and particular information about the categories themselves.

\begin{remark}
For those uninterested in the wheeled case, either of Sections \ref{sec:EZ} or \ref{sec:weaksegalproperads} can be read immediately (omitting Remark \ref{failure_for_gw} and Theorem \ref{gammaw not ez}), assuming the reader is familiar with Chapters 6 and 7 of \cite{hry}.
The reader primarily interested in the wheeled case should read Sections \ref{modification section}, \ref{sec:reedygammaw}, and \ref{sec:weaksegalproperads}, in that order. Such a reader could, if interested, conclude by reading the relevant parts of Section \ref{sec:EZ}: (the statement of) Lemma \ref{strong pushouts}, Remark \ref{failure_for_gw}, Definition \ref{def:EZ cat}, and Theorem \ref{gammaw not ez}.
\end{remark}

Notation, terms, and definitions may all be found in \cite{hry}, which boasts a thorough index and a comprehensive list of notation, and which is freely available on the arxiv.

\section{A modification to the definition of \texorpdfstring{$\Gammaw$}{Γ↻}}\label{modification section}

We address the definition of the wheeled graphical category $\Gammaw$ from \cite[Chapter 9]{hry}, whose objects were defined as wheeled properads freely generated by graphs.
With this convention, the two exceptional connected graphs $\uparrow$ and $\wheel$
generate the same wheeled properad (see \cite[Example 9.17]{hry}) despite having distinct combinatorial structures.
We now take the point of view that the objects of $\Gammaw$ should be regarded as the graphs themselves, rather than the wheeled properads that they generate.

Recall from \cite[Definition 9.59]{hry} that a \textbf{wheeled properadic graphical map}, or simply a \textbf{graphical map}, is defined as a wheeled properad map $G \xrightarrow{f}  K$ between graphical wheeled properads such that the map $f(G) \to K$ is a subgraph.
This definition follows our usual convention of writing, for $H\in \gwheelc$, the symbol `$H$' as a stand-in for the wheeled properad $\Gammaw(H)$.

This definition is dependent on the subgraph structure of $K$. 
As it stands, $\uparrow$ and $\wheel$ have a different set of subgraphs: the former has only itself as a subgraph, while the latter has both $\uparrow$ and $\wheel$ as subgraphs (see \cite[Corollary 9.53]{hry}).
Thus one expects, from the previous definition, that $\hom(\wheel, \uparrow) = \varnothing$ while $\hom(\uparrow, \uparrow) = *$.
Consequently, ${\wheel} \ncong {\uparrow}$.

This is not the convention that we used in \cite{hry}, where we considered that the subgraph inclusion ${\uparrow} \to {\wheel}$ to be an actual isomorphism in $\Gammaw$, since this morphism becomes an isomorphism in $\wproperad$.
See, for instance, \cite[Remark 9.67]{hry}.
At first glance, this seems to be quite important in the proof of the nerve theorem\footnote{By which we mean that the nerve functor is fully-faithful with essential image characterized by a Segal condition \cite[10.33 \& 10.35]{hry}.};
for example, the proof of \cite[Lemma 10.40]{hry} 
for $G={\wheel}$ uses that ${\uparrow} \to {\wheel}$ is an isomorphism. We will explain in Section \ref{SS:segal} why this is not actually necessary. 

For the purposes of this note, we will regard the combinatorial structure, rather than the categorical structure, as more central to the definitions.
Let the category $\oldB$ have $\Ob(\oldB) = \gwheelc$ and 
\[ \oldB(G,K) \subseteq \wproperad(\Gammaw(G),\Gammaw(K))\]
to be the set of wheeled properadic graphical maps.
In particular, the unique map ${\uparrow} \to {\wheel}$ is not an isomorphism in $\oldB$; in fact, there are no non-identity maps with source $\wheel$.
The nerve theorem for 
this $\oldB$ needs a bit of adjustment to definitions, which we outline in Section \ref{SS:segal}. The basic issue is as follows: 
Since there are no non-identity maps with source $\wheel$, we could replace any $\sK$ with a $\sK$ satisfying $\sK(\wheel) = \varnothing$ but otherwise the same. If our notion of Segal maps does not encompass ${\uparrow} \to {\wheel}$, then we may not have
\[ \sK(\wheel) \xrightarrow{\eta_\wheel} (N\sP_{\sK})(\wheel) \cong (N\sP_{\sK})(\uparrow) \xleftarrow{\cong} \sK(\uparrow) \]
is an isomorphism.

An alternative is to rid ourselves of $\wheel$ entirely, and define $\oldC$ to be the full subcategory of $\oldB$ with object set $\gwheelc \setminus \{ \wheel \}$.
Then $\oldC$ does admit a nerve theorem without any real changes to the text of \cite{hry}, except for ignoring special cases and clauses related to the exceptional loop $\wheel$. 
The fact that Theorem \ref{wproperadnerve} holds for both $\oldB$ and $\oldC$ indicates that strict notions of wheeled properads are insensitive to the choice of $\Gammaw$. 
However, when defining weak versions (as in \cite[10.20(1)]{hry} or Definition \ref{weak segal properad}), the answer changes depending on the choice between $\oldB$ or $\oldC$. 
We can imagine contexts where one choice is more natural than the other, so have elected to include both cases in this paper.
We note that Theorem \ref{gammaw reedy} holds for both of the categories $\oldB$ and $\oldC$.


\subsection{Segal core and the nerve theorem}\label{SS:segal}

Throughout this section, let $\Gammaw$ denote either of the categories $\oldB$ or $\oldC$.
The following is an extension of \cite[Definition 10.30]{hry}, which only defined Segal cores for ordinary graphs having at least one internal edge. 

\begin{definition}
Suppose that $G$ is in $\Gammaw$. 
\begin{enumerate}
\item
Let $B_G$ be the category whose set of objects is $\vertex(G) \amalg \edge(G)$ with non-identity morphisms
\begin{align*}
e &\overset{\inp}\longrightarrow v & \text{when $e$ is an input of $v$}\\
e &\overset{\out}\longrightarrow v & \text{when $e$ is an output of $v$.}
\end{align*}
There is a functor $F \colon B_G \to \gwheelcset$ with
\begin{align*}
	F(e) &= \Gammaw[\uparrow] & F(e \overset{\inp}\longrightarrow v) &= (\Gammaw[\uparrow] \xrightarrow{\inp_e} \Gammaw[C_v]) \\
	F(v) &= \Gammaw[C_v] & F(e \overset{\out}\longrightarrow v) &= (\Gammaw[\uparrow] \xrightarrow{\out_e} \Gammaw[C_v])
\end{align*}
Define the \textbf{wheeled properadic Segal core} $\Sc[G] \in \gwheelcset$ as the colimit of the functor $F$.
\item
Denote by
\[
	\Sc[G] \xrightarrow{\epsilon_G}  \Gammaw[G]
\]
the map induced by the subgraph inclusions
\begin{align*}
	C_v &\to G \\
	\uparrow_e &\to G,
\end{align*}
and call it the \textbf{wheeled properadic Segal core map}.
\end{enumerate}
\end{definition}

If $G$ is an ordinary graph with at least one internal edge, then this definition, for both $\Sc[G]$ and $\epsilon_G$, coincides with that of \cite[Definition 10.30]{hry}. 
The current definition covers all graphs, with the following new definitions: 
\begin{center}
\begin{tabular}{@{}llr@{}}
	\toprule 
$G$ & $\Sc[G]$ & $\epsilon_G$ \\
	\midrule
$\uparrow$ & $\Gammaw[\uparrow]$ &  identity \\
$\wheel$ & $\Gammaw[\uparrow]$ & $\Gammaw[\uparrow] \hookrightarrow \Gammaw[\wheel]$ \\
$C_{(n;m)}$ & $\Gammaw[C_{(n;m)}]$ & identity \\
\bottomrule
\end{tabular}
\end{center}

We should also alter \cite[Definition 10.26]{hry} by letting the corolla ribbon be a \emph{limit} over $B_G$. This guarantees that $\sK(G)_1 \cong \gwheelcset\left(\Sc[G], \sK\right)$ and that \cite[Lemma 10.31]{hry} holds.
With the expanded definition of wheeled properadic Segal map, we can ask for the following.

\begin{definition}
	We say $\sK$ satisfies the \textbf{wheeled properadic Segal condition} if the wheeled properadic Segal map 
\[
	\sK(G) \cong \gwheelcset\left(\Gammaw[G], \sK\right) \to \gwheelcset\left(\Sc[G], \sK\right)
\]
	is a bijection for every $G \in \gwheelc$.
\end{definition}

By the table above, we see that this bijection is automatic for every $\sK$ when $G = {\uparrow}$ and $G = C_{(n;m)}$.
In \cite{hry}, $\epsilon_{\wheel}$ is an isomorphism, hence $\epsilon^*_{\wheel}$ is a bijection for all $\sK$. If we are working with the category $\oldC$, we do not have $\wheel$ as an object.
Thus this definition only has new content when $\Gammaw = \oldB$ and $G = {\wheel}$.

We note that for both $\Gammaw = \oldB$ or $\Gammaw = \oldC$ the \emph{nerve theorem} holds:

\begin{theorem}[{Theorem 10.33 of \cite{hry}}]
\label{wproperadnerve}
Suppose $\sK \in \gwheelcset$.  Then the following statements are equivalent.
\begin{enumerate}
\item
There exist a wheeled properad $\sP$ and an isomorphism $\sK \cong N\sP$.
\item
$\sK$ satisfies the wheeled properadic Segal condition.
\item
$\sK$ is a strict $\infty$-wheeled properad.
\end{enumerate}
\end{theorem}

In fact, from Section 10.2.3 until the end of Chapter 10, all theorems and proofs hold for $\oldB$ and $\oldC$ \emph{mutatis mutandis}.
We only note that if $d\colon \bullet \to {\wheel}$ is the exceptional inner coface map, then the inner horn inclusion $\Lambda^d[\wheel] \to \Gammaw[\wheel]$ is exactly the Segal core inclusion $\epsilon_{\wheel} \colon {\Gammaw[\uparrow]} \to \Gammaw[\wheel]$.
The condition that $\sK$ admits unique fillers for $\Lambda^d[\wheel] \to \Gammaw[\wheel]$ is precisely the condition that $\sK(\wheel) \to \sK(\uparrow)$ admits a unique section, which means $\epsilon_\wheel^* \colon \sK(\wheel) \to \sK(\uparrow)$ is a bijection.

\section{A Generalized Reedy Structure on \texorpdfstring{$\Gammaw$}{Γ↻}}
\label{sec:reedygammaw}

In this section we outline the generalized Reedy structure on $\Gammaw$, taken to mean one of the two categories $\oldB$ or $\oldC$ from the previous section. 

We recall the definition from \cite{bm}.

\begin{definition}
\label{def:greedy}
A \textbf{generalized Reedy structure} \index{generalized Reedy structure} on a small category $\mathcal{R}$ consists of
\begin{itemize}
\item
wide subcategories $\mathcal{R}^+$ and $\mathcal{R}^-$, and 
\item
a degree function $\deg\colon \Ob(\mathcal{R}) \to \mathbb{N}$
\end{itemize}
satisfying the following four axioms.
\begin{enumerate}[(i)]
\item 
Non-invertible morphisms in $\mathcal{R}^+$ (resp., $\mathcal{R}^-$) raise (resp., lower) the degree.  Isomorphisms in $\mathcal{R}$ preserve the degree.
\item 
$\mathcal{R}^+ \cap \mathcal{R}^- = \Iso(\mathcal{R})$.
\item
Every morphism $f$ of $\mathcal{R}$ factors as $f = gh$ with $g  \in \mathcal{R}^+$ and $h \in \mathcal{R}^-$, and this
factorization is unique up to isomorphism.
\item 
If $\theta f=f$ for $\theta \in \Iso(\mathcal{R})$ and $f\in \mathcal{R}^-$, then $\theta$ is an identity.
\end{enumerate} If, morever, the condition
\begin{enumerate}[(iv')]
\item If $f \theta=f$ for $\theta \in \Iso(\mathcal{R})$ and $f\in \mathcal{R}^+$, then $\theta$ is an identity
\end{enumerate}
holds, then we call this a generalized \textbf{dualizable} \index{dualizable} Reedy structure.
\end{definition}

\begin{definition}\label{def gammaw reedy structure}
Define the \textbf{degree} of a graph $G \in \gwheelc$ to be 
\[ \deg (G) = \begin{cases}
0 & \text{if }G = \uparrow; \\
1 & \text{if }G = \bullet; \\
	 |\vertex(G)| + |\edgei(G)| + 1 & \text{otherwise.}
\end{cases}\]
In \cite[9.4.5]{hry}, we defined two wide subcategories of $\Gammaw$ as follows:
\begin{itemize}
\item $\Gammaw^+$ \label{note:gammawplus} is the subcategory generated by all isomorphisms and coface maps
\item $\Gammaw^-$ is the subcategory generated by all isomorphisms and codegeneracy maps.
\end{itemize}
\end{definition}

We will now give a better characterization of these two subcategories, in terms of direct properties of the maps involved.
Recall from \cite[Lemma 9.23, Definition 9.59]{hry} that every map $f \colon G \to K$ in $\Gammaw$ consists of two functions: $f_0$ with domain $\edge(G)$ and $f_1$ with domain $\vertex(G)$.

\begin{proposition}
\label{plus characterization}
Suppose that $f\colon G \to K$ is in $\Gammaw$.
Then $f\in \Gammaw^+$ if and only if for each $v\in \vertex(G)$, the subgraph $f_1(v)$ is not an edge.
\end{proposition}
\begin{proof}
If $f\colon G \to K$ and $g \colon K \to H$ satisfy the condition then so does $g\circ f$.
Furthermore, each coface map and isomorphism satisfies the condition, hence $\Gammaw^+$ is contained in the class of maps satisfying the condition.

In the other direction, suppose that $f\colon G \to K$ satisfies the condition that for each $v\in \vertex(G)$, the subgraph $f_1(v)$ is not an edge. 
Then in the decomposition of \cite[Theorem 9.69]{hry}, the map $\sigma$ is the identity. Thus $f = h \circ \delta \circ i \in \Gammaw^+$.
\end{proof}

\begin{proposition}
\label{minus characterization}
Consider the class $\mathcal C$ of morphisms $f\colon G \to K$ in $\Gammaw$ which satisfy the following conditions:
\begin{enumerate}
	\item For each $v\in \vertex(G)$, the subgraph $f_1(v)$ is either an edge or a corolla of $K$. \label{minus one}
	\item Each corolla $C_w$ of $K$ appears as $f_1(v)$ for some $v\in \vertex(G)$. \label{minus two}
	\item The graph $G$ has at least one vertex. \label{minus three}
\end{enumerate}
A map $f$ is in $\Gammaw^-$ if and only if $f$ is an isomorphism or $f\in \mathcal C$.
\end{proposition}
\begin{proof}
Let $f \colon G\to K$ be a morphism in $\Gammaw$ where $G$ does not have a vertex. Then $G\in \{\uparrow, \wheel \}$; since we have
\[
	\Gammaw^-(G,K) = \begin{cases}
		\varnothing & G\neq K \\
		\{ \Id_G \} & G = K
	\end{cases}
\]
we know that $f\in \Gammaw^-$ if and only if $f$ is an identity.
For the remainder of the proof, we only consider maps $f \colon G\to K$ so that $G$ has at least one vertex.

The class $\mathcal C$ is closed under composition.
Further, codegeneracies are in $\mathcal C$ and isomorphisms not involving $\uparrow$ or $\wheel$ are in $\mathcal C$, hence 
$\Gammaw^-(G,K) \subseteq \mathcal C$ whenever $G\notin \{ \uparrow, \wheel \}$.

Conversely, suppose that $f\colon G \to K$ is in $\mathcal C$. In the decomposition given in \cite[Theorem 9.69]{hry}, 
\begin{itemize}
	\item condition \eqref{minus one} gives that $\delta = \Id$, while 
	\item conditions \eqref{minus two} and \eqref{minus three} give that $h= \Id$.
\end{itemize}
Thus $f = i \circ \sigma \in \Gammaw^-$.
\end{proof}

We now turn to proving Theorem \ref{gammaw reedy}. To be precise, we will show that the degree function and wide subcategories from Definition \ref{def gammaw reedy structure} constitute a generalized Reedy structure on the category $\Gammaw$. We will use this structure for the remainder of the section.

\begin{proposition}
\label{reedy1}
The graphical category $\Gammaw$ satisfies condition (i) in Definition \ref{def:greedy}.
\end{proposition}

\begin{proof}
Isomorphisms preserve the degree, so it is enough to show that every coface map strictly increases degree and every codegeneracy map strictly decreases degree.

There are no coface maps with target $\uparrow$ or $\bullet$.
We know that if $G\notin \{\uparrow, \bullet\}$ then $\deg(G) \geq 2 > \deg(\bullet) = 1 > \deg(\uparrow) = 0$ , so any coface map with source $\bullet$ or $\uparrow$ (including the exceptional inner coface \cite[9.30]{hry}) strictly increases degree.

We now turn to coface maps $f\colon G \to K$ with $G, K\notin \{\uparrow, \bullet\}$.
Dioperadic coface maps are nondecreasing on $|\edgei(-)|$ and strictly increasing on $|\vertex(-)|$. Contracting coface maps are strictly increasing on $|\edgei(-)|$ and constant on $|\vertex(-)|$. Thus any such coface map $f$ strictly increases degree.

There is no codegeneracy map with target $\bullet$ and the codegeneracy map with target $\uparrow$ decreases degree by two. There are no codegeneracy maps with source $\bullet$ or $\uparrow$.
Each codegeneracy map is nonincreasing on $|\edgei(-)|$ and is strictly decreasing on $|\vertex(-)|$. Thus, in all cases, codegeneracy maps strictly decrease degree.
\end{proof}

\begin{remark}
Hidden in the proof of the preceding proposition is an indication of why we cannot use a simpler degree function
\[
	\deg'(G) = |\vertex(G)| + |\edgei(G)|
\]
for our Reedy structure: 
the exceptional inner coface map $\bullet \to {\wheel}$ preserves $\deg'$ but is not invertible.
This issue disappears if one choses to work with $\oldC$ rather than $\oldB$, as $\wheel$ is not an object of $\oldC$.
\end{remark}

\begin{proposition}
\label{prop:wheeled_decomposition}
The graphical category $\Gammaw$ satisfies condition (iii) in Definition \ref{def:greedy}.  In other words, every map in $f\in \Gammaw$ factors as
\[
f=gh,
\]
where $h\in \Gammaw^-$ and $g\in \Gammaw^+$, and this factorization is unique up to isomorphism.
\end{proposition}

\begin{proof}
Existence of such a factorization follows at once from \cite[Lemma 9.72]{hry} and \cite[Proposition 9.75]{hry}. 
Suppose that $f=gh=g'h'$ are two such decompositions, with $g,g'\in \Gammaw^+$ and $h,h'\in \Gammaw^-$. Factor $g=ab$, $g'=a'b'$ as in \cite[Proposition 9.75]{hry}, with $a,a' \in \Gammawout \subset \Gammaw^+$ and $b,b' \in \Gammawimage$. Then $f=a(bh) =a'(b'h')$ are factorizations of $f$ as in \cite[Proposition 9.75]{hry}, hence uniqueness gives us an isomorphism $i$ so that $a'i = a$ and $ibh = b'h'$. But then $(ib)h = b'h'$ are two decompositions as in \cite[Lemma 9.72]{hry}, so by uniqueness there is an isomorphism $i'$ with $b'i' = ib$ and $i'h=h'$. But then $g'i' = (a'b')i' = a'(ib) = ab = g$, so the diagram
\[
\begin{tikzcd}
G \rar{h}	\dar[swap]{h'} & H \dar{g} \arrow{dl}{i'}[swap]{\cong}  \\
H' \rar[swap]{g'} & K
\end{tikzcd}
\]
commutes. 
\end{proof}

\begin{proposition}
\label{P:cond2WHEELED}  
The graphical category $\Gammaw$ satisfies condition (ii) in Definition \ref{def:greedy}, namely
\[
\Gammaw^+ \cap \Gammaw^- = \Iso(\Gammaw).
\]
\end{proposition}

\begin{proof}
        Inclusion from right to left is straightforward. For the reverse, suppose that $f\colon G\to K$ is in $\Gammaw^+ \cap \Gammaw^-$. 
If $\vertex(G) = \varnothing$, then $f$ is an isomorphism by Proposition \ref{minus characterization}.
        By Proposition \ref{prop:wheeled_decomposition} we know that, since $f\in \Gammaw^+$, $f$ admits a factorization
\[
f = \partial i,
\]
where $\partial$ is a composition of coface maps and $i$ is an isomorphism.  In particular, we have
\begin{equation*}
\label{GleK} \tag{$\ast$}
\deg(G) \leq \deg(K).
\end{equation*}
Since $f\in \Gammaw^-$, this same proposition also gives a factorization
\[
f= i' \sigma,
\]
where $\sigma$ is a composition of codegeneracies and $i'$ is an isomorphism.  So we have
\[
\deg(G) \geq \deg(K).
\]
Together with \eqref{GleK}, we conclude that 
\[
\deg(G) = \deg(K).
\]
This implies that $\partial$ is the identity, so $f=i$ is an isomorphism.
\end{proof}

Given a graph $G$, consider the following two sets:
\begin{align*}
	S^{\inp}_G &= \edgei(G) \amalg \inp(G) \\
	S^{\out}_G &= \edgei(G) \amalg \out(G)
\end{align*}

\begin{lemma}\label{f zero iso}
Suppose that $f,f' \colon G \to K$ are isomorphisms in $\Gammaw$. 
The following are equivalent:
\begin{enumerate}
	\item $f=f'$ \label{item f}
	\item $f_0 = f'_0$ \label{item f0}
	\item $f_0|_{S^{\inp}_G} = f'_0|_{S^{\inp}_G}$ and $f_0|_{S^{\out}_G} = f'_0|_{S^{\out}_G}$. \label{item f0 pieces}
\end{enumerate}
\end{lemma}
\begin{proof}
It is clear that \eqref{item f} implies \eqref{item f0 pieces} even when $f$ is not an isomorphism. 

Suppose that $f$ is an isomorphism. Then $f$ is bijective on edges and sends each vertex $v$ to a corolla $f(v) = C_{v'}$. This vertex $v'$ is uniquely determined by its profiles $(f_0(\inp(v)); f_0(\out(v))$, hence by $f_0$. This shows that \eqref{item f0} implies \eqref{item f}.
Since isomorphisms send inputs to inputs, outputs to outputs, and inner edges to inner edges, we have that \eqref{item f0 pieces} implies \eqref{item f0}.
\end{proof}

\begin{proposition}
\label{reedy4}
The graphical category $\Gammaw$ satisfies condition (iv) 
in Definition \ref{def:greedy}.  In other words, if $f\in \Gammaw^-$, $\theta \in \Iso(\Gammaw)$, and $\theta f =f$, then $\theta = \Id$.
\end{proposition}

\begin{proof}
        Since isomorphisms in $\Gammaw$ are determined by their actions on edge sets by the previous lemma, 
        it is enough to show that $\theta_0$ is an identity. But now this comes down to the same fact in $\Set$: If $f_0$ is surjective and $\theta_0f_0 = f_0$, then $\theta_0 = \Id$.
\end{proof}

\begin{lemma}\label{which parts injective}
If $f\colon G \to K$ is in $\Gammaw^+$, then $f_0|_{S^{\inp}_G}$ and $f_0|_{S^{\out}_G}$ are injective.
\end{lemma}
\begin{proof}
We will show that $f_0|_{S^{\inp}_G}$ is injective, since the proof for $f_0|_{S^{\out}_G}$ is identical. 

Notice that if $f$ is a coface map or an isomorphism, then $f_0|_{S^{\inp}_G}$ is injective.
Suppose that $f\colon G \to K$ and $g\colon K \to H$ are in $\Gammaw^+$ and that $f_0|_{S^{\inp}_G}$ and $g_0|_{S^{\inp}_K}$ are injective. We wish to show that $(g\circ f)_0|_{S^{\inp}_G}$ is injective as well, since that will imply the result for every map in $\Gammaw^+$.
If $G$ has no vertices, then $S^{\inp}_G = *$ and it follows that $(g\circ f)_0|_{S^{\inp}_G}$ is injective.
Assume then that $G$ has at least one vertex.
It is enough to show that $f_0|_{S^{\inp}_G}$ factors through $S^{\inp}_K$: 
\[ \begin{tikzcd}
S^{\inp}_G \dar[hook] \rar[dotted] & S^{\inp}_K\dar[hook]  \\
\edge(G) \rar{f_0} & \edge(K) \rar{g_0} & \edge(H).
\end{tikzcd} \]

We use Proposition \ref{plus characterization}.
If $e$ is an inner edge of $G$, $e\in \inp(v) \cap \out(w)$, then $f_0(e) \in \inp(f_1(v)) \cap \out(f_1(w))$. Since $f_1(v)$ and $f_1(w)$ each contain a vertex and are connected, $f_0(e)$ is an inner edge.
If $e \in \inp(G)$ is an input edge, then since $G$ was assumed to have a vertex, $e$ is adjacent to a vertex $v$. Then since $f_1(v)$ is not an edge, $f_0(e) \in \inp(f_1(v))$, hence $f_0(e) \notin \out(K)$.
In both cases, $f_0(e) \in S^{\inp}_K$, so $f_0|_{S^{\inp}_G}$ factors through $S^{\inp}_K$.
\end{proof}

\begin{proposition}
\label{reedy4prime}
The graphical category $\Gammaw$ satisfies condition (iv') 
in Definition \ref{def:greedy}.  In other words, if $f\in \Gammaw^+$, $\theta \in \Iso(\Gammaw)$, and $f \theta =f$, then $\theta = \Id$.
\end{proposition}

\begin{proof}
Write $f\colon G \to K$, $\theta\colon G \to G$.
Write $\theta_0|_{S^{\inp}_G}$ for the induced bijection on $S^{\inp}_G$.
By assumption that $f\theta = f$, we have
\[
	f_0|_{S^{\inp}_G} \circ \theta_0|_{S^{\inp}_G} = f_0|_{S^{\inp}_G}.
\]
By Lemma \ref{which parts injective}, $f_0|_{S^{\inp}_G}$ is injective, hence $\theta_0|_{S^{\inp}_G} = \Id_{S^{\inp}_G}$.
Similarly, since $f_0|_{S^{\out}_G}$ is injective, 
$\theta_0|_{S^{\out}_G} = \Id_{S^{\out}_G}$ and Lemma \ref{f zero iso} implies that $\theta = \Id$.
\end{proof}

\begin{proof}[Proof of Theorem \ref{gammaw reedy}]
Combine Propositions \ref{reedy1}, \ref{prop:wheeled_decomposition}, \ref{P:cond2WHEELED}, \ref{reedy4}, and \ref{reedy4prime}.
\end{proof}

\begin{remark}
The subcategory inclusion $\Gamma \subseteq \Gammaw$ respects the generalized Reedy structures on both sides in the sense that $\Gamma^- \subseteq \Gammaw^-$ and $\Gamma^+ \subseteq \Gammaw^+$.
Further, we have $\deg_\Gamma (G) \leq \deg_{\Gammaw} (G)$ for all wheel-free graphs $G$ (this inequality is usually strict), so the inclusion functor preserves filtrations.
More importantly, we have the following:
\begin{itemize}
\item If $G$ is wheel-free and $G\to H$ is in $\Gammaw^-$, then $H$ is also wheel-free and $G\to H$ is in $\Gamma^-$.
\item If $H$ is wheel-free and $G\to H$ is in $\Gammaw^+$, then $G$ is also wheel-free and $G\to H$ is in $\Gamma^+$.
\end{itemize}
These facts imply that $\Gamma \to \Gammaw$ is fibering and cofibering in the sense of \cite{hv15}.
In the case of strict Reedy categories, Hirschhorn and Voli\'c show that a Reedy functor $\mathcal{R} \to \mathcal{S}$ is fibering (resp. cofibering) if and only if $\mathcal{M}^{\mathcal{S}} \to \mathcal{M}^{\mathcal{R}}$ is a right (resp. left) Quillen functor \cite{hv15}.
If the forward direction of this result can be extended to the setting of generalized Reedy model structures, then the functor $\mathcal{M}^{\Gammaw^{\op}} \to \mathcal{M}^{\Gamma^{\op}}$ would be both left and right Quillen.
\end{remark}

\section{An Eilenberg-Zilber structure on \texorpdfstring{$\Gamma$}{Γ}}\label{sec:EZ}

The goal of this section is to prove Theorem \ref{gamma ez}, that is, to show that $\Gamma$ is an EZ-category in the sense of Berger and Moerdijk (see Definition \ref{def:EZ cat}). 
To that end, we need a solid understanding of codegeneracies and their sections. 
The key result is Lemma \ref{eilenbergzilber}.

We will prove a graphical version of the Eilenberg-Zilber lemma.  The following preliminary observation is needed.  It says that every codegeneracy map in the graphical category has a section, which can furthermore be chosen to have any pre-selected edge in its image.  

The idea of the construction is given in two figures: consider the codegeneracy $s$ as operating on a portion of the graph in Figure \ref{degeneracy_picture}, with two possible sections as given in Figure \ref{sections_picture}.
If $w$ or $u$ is not present (that is, if $b_1$ is an input or $b_{-1}$ is an output), the definition of sections is slightly simpler.

\begin{figure}
\begin{center}
   \includegraphics[width=.4\textwidth]{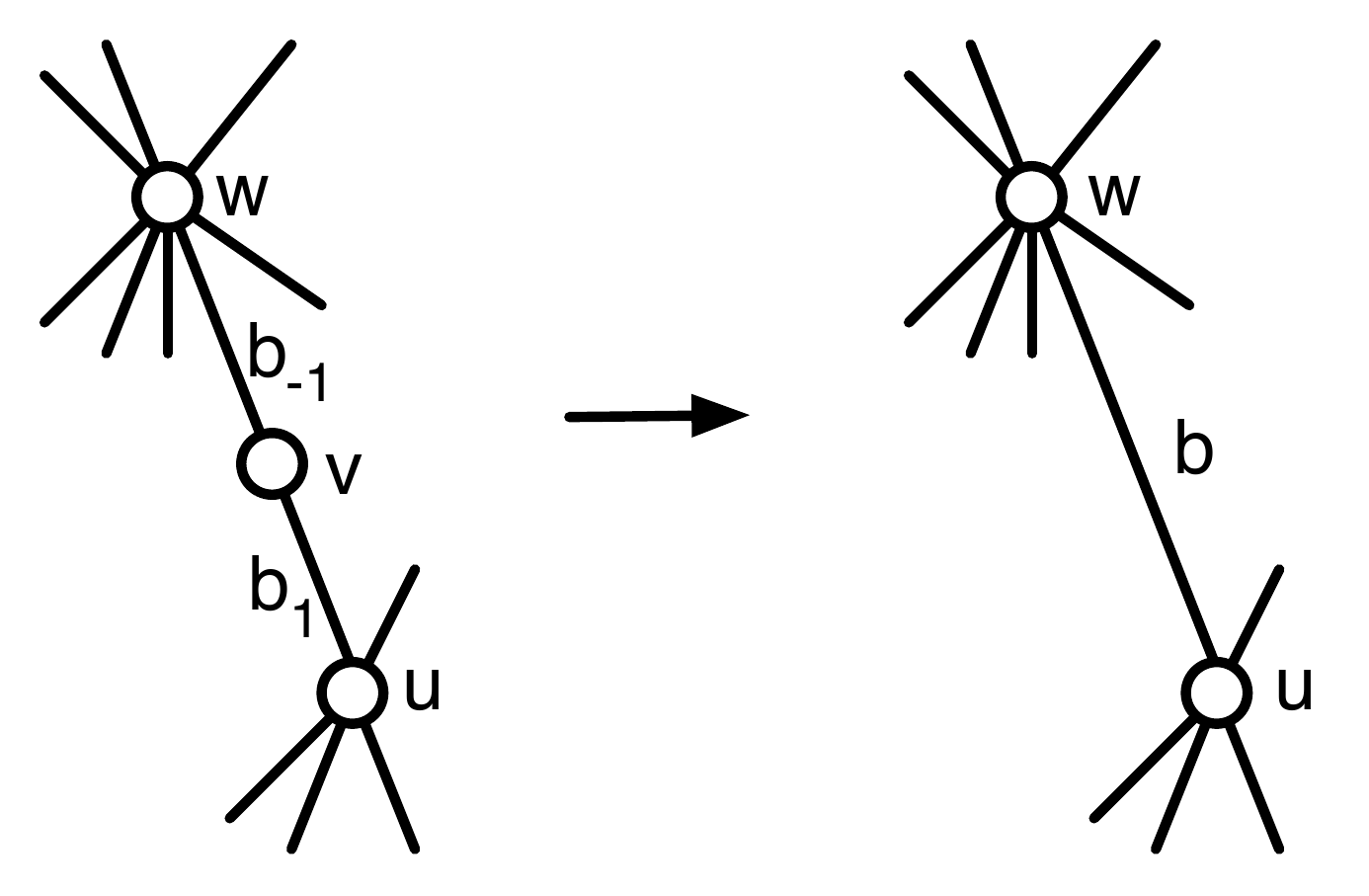}
\end{center}
	\caption{Local picture of a codegeneracy $s$}\label{degeneracy_picture}
\end{figure}

\begin{figure}
\begin{center}
	\includegraphics[width=.4\textwidth]{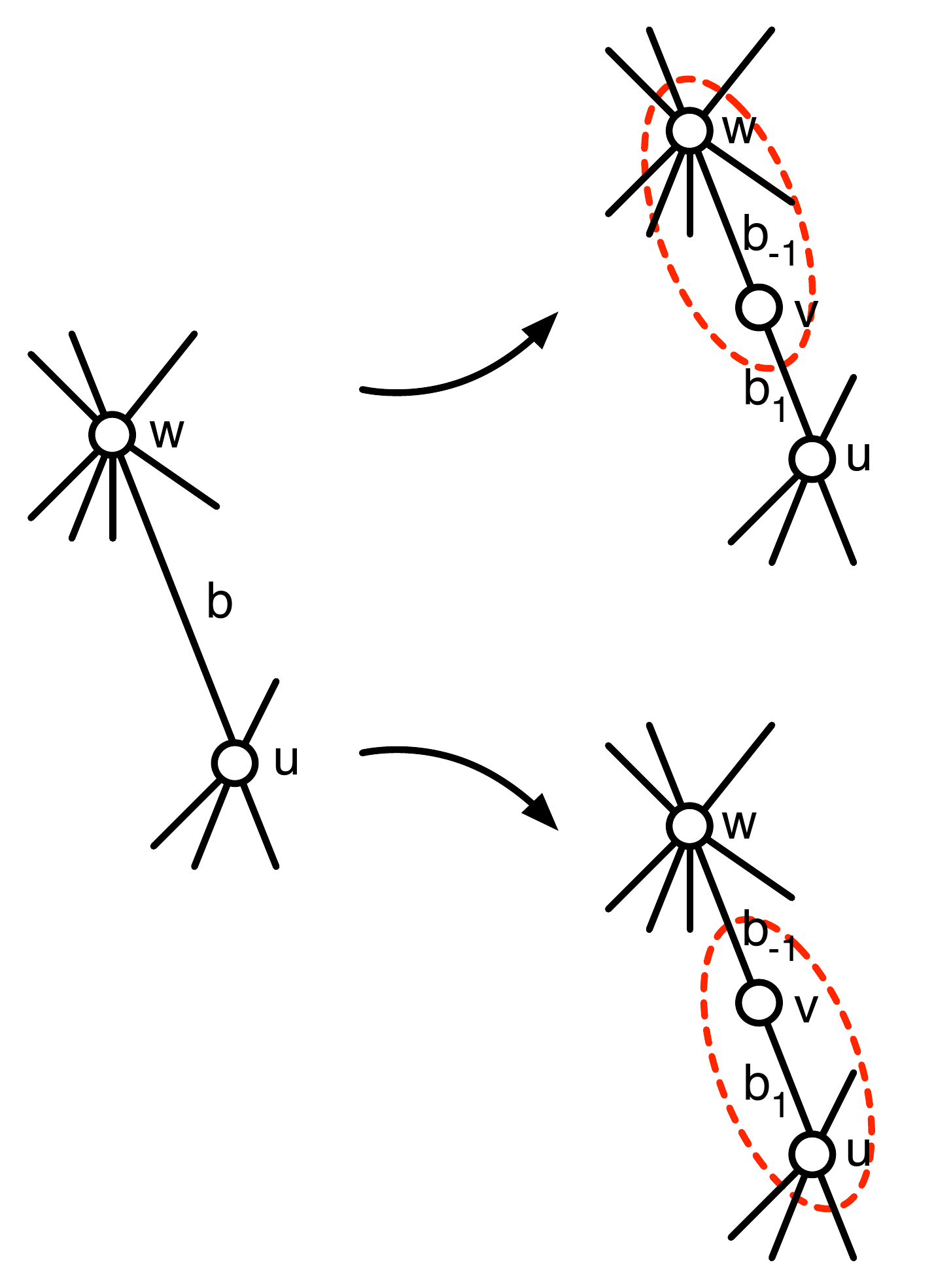}
\end{center}
	\caption{Two sections of $s$ from Figure \ref{degeneracy_picture}, corresponding to $e=b_1$ (top) and $e=b_{-1}$ (bottom).}\label{sections_picture}
\end{figure}

\begin{lemma}
\label{codegensection}
Suppose $s \colon \Gamma(G) \to \Gamma(H)$ is a codegeneracy and $e \in \edge(G)$.  Then there exists a coface map $f \colon \Gamma(H) \to \Gamma(G)$ such that
\begin{enumerate}
\item
$f$ is a section of $s$ (i.e., $sf = \Id_H$), and
\item
$e$ is in the image of $f \colon \edge(H) \to \edge(G)$.
\end{enumerate}
\end{lemma}

\begin{proof}
The codegeneracy map $s$ corresponds to a degenerate reduction $H = G(\uparrow)$, in which the exceptional edge $\uparrow$ is substituted into a vertex $v \in G$, and a corolla is substituted into every other vertex of $G$.  The vertex $v$ must have precisely one incoming edge $b_1$ and one outgoing edge $b_{-1}$. 
The set of edges of $H$ is the quotient
\[
\edge(H) = \dfrac{\edge(G)}{\left(b_{-1} \sim b_1\right)},
\]
and we write $b \in \edge(H)$ for the common image of $b_{-1}$ and $b_1$ under $s$.   The codegeneracy map $s$ is then given on edges by \[
s(a) = 
\begin{cases}
a & \text{ if $a \in \edge(G) \setminus \{b_{-1}, b_1\}$},\\
b & \text{ if $a \in \{b_{-1}, b_1\}$}
\end{cases}
\]
and on vertices by \[
s(u) = 
\begin{cases}
C_u & \text{ if $u \in \vertex(G) \setminus \{v\}$},\\
\uparrow_b & \text{ if $u = v$},
\end{cases}
\] where $C_u$ denotes the corolla with the same profiles as the vertex $u$. 
In particular, we have
\[
\vertex(H) = \vertex(G) \setminus \{v\}.
\]

Now we define $f \colon \Gamma(H) \to \Gamma(G)$ as follows.  For $a \in \edge(H)$, we define
\[
f(a) = 
\begin{cases}
a & \text{ if $a \not= b$},\\
b_{-1} & \text{ if $a=b$ and $e \not= b_1$},\\
b_1 & \text{ if $a=b$ and $e=b_1$}.
\end{cases}
\]
Suppose $x \in H$ is a vertex.
\begin{enumerate}
\item
Define $f(x) = C_x$ if $x$ is not adjacent to $b$.
\item
Suppose $x$ is the terminal vertex $w$ of $b$ (if such exists).  Then define
\[
f(x) = 
\begin{cases}
C_w & \text{ if $e \not= b_1$},\\
C_w \circ_{b_{-1}} C_v & \text{ if $e = b_1$},
\end{cases}
\]
where $C_w \circ_{b_{-1}} C_v$ is the subgraph of $G$ spanned by the vertices $w$ and $v$ with internal edge $b_{-1}$ (see \cite[Example 2.17]{hry}).
\item
Suppose $x$ is the initial vertex $u$ of $b$ (if such exists).  Then define
\[
f(x) = 
\begin{cases}
C_v \circ_{b_1} C_u & \text{ if $e \not= b_1$},\\
C_u & \text{ if $e = b_1$},
\end{cases}
\]
where $C_v \circ_{b_1} C_u$  is the subgraph of $G$ spanned by the vertices $v$ and $u$ with internal edge $b_1$.
\end{enumerate}
By construction, the (inner or outer) coface map $f$ has $e$ in its image.
Moreover, $f$ is a section of $s$ because substituting the exceptional edge $\uparrow$ into the vertex $v$ in $C_w \circ_{b_{-1}} C_v$ (resp., $C_v \circ_{b_1} C_u$) yields the corolla $C_w$ (resp., $C_u$).
\end{proof}

\begin{remark}\label{failure_for_gw}
	The analogue of Lemma \ref{codegensection} fails for $\Gammaw$.
	As an example, let $G = \xi_1^1C_{(1;1)}$ be the contracted corolla with one vertex and one (internal) edge. 
	Then $G(\uparrow) = {\wheel}$ gives the codegeneracy map $G \to {\wheel}$, but there are no maps at all from $\wheel$ to $G$.
	Another class of examples (which are also present in $\oldC$, rather than just $\oldB$) are given by gluing $C_{(n;m)}$ to $C_{(1;1)}$ and mapping to $\xi_i^jC_{(n;m)}$ as in Figure \ref{gwheel degen}.
\begin{figure}
	\begin{center}
   \includegraphics[width=.5\textwidth]{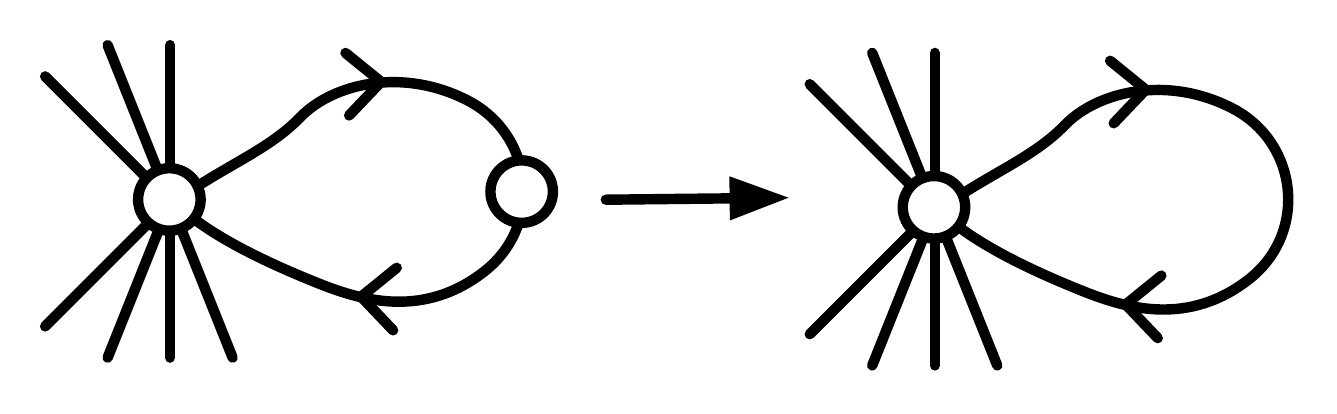}
   \end{center}
	\caption{A codegeneracy in $\Gammaw$ without a section}\label{gwheel degen}
\end{figure}
\end{remark}

Consider the function
\begin{align*}
	\Gamma(G, H) &\to \Set(\edge(G), \edge(H)) \\
	f &\mapsto f_0
\end{align*}
which is a monomorphism by \cite[Corollary 6.62]{hry}, i.e., $(-)_0$ is the morphism part of the (faithful) functor $\Gamma \to \Set$ which sends $G$ to $\edge(G)$.

If $f\colon X \to Y$ is a map in some category $\mathcal C$, write \[ \sections(f) = \{ g \, | \, fg = \Id_Y \} \subset \mathcal C(Y,X)\] for the set of sections of $f$.
Any functor $F$ out of $\mathcal C$ induces a function $\sections(f) \to \sections(Ff)$ by sending $g$ to $Fg$.

\begin{proposition}\label{sections_bijection}
	If $s\in \Gamma^-$, then  $\sections(s) \to \sections(s_0)$ is a bijection.
\end{proposition}
\begin{proof}
	Write $s \colon G_n \to G_{n-m}$ where the subscript denotes the degree of the graph (ie, the number of vertices).
	We have a commutative square
	\[ \begin{tikzcd}
	\sections(s) \rar[hook] \dar & \Gamma(G_{n-m}, G_n) \dar[hook]{\text{\cite[Corollary 6.62]{hry}}} \\
	\sections(s_0) \rar[hook] & \Set(\edge(G_{n-m}), \edge(G_n))
	\end{tikzcd} \]
	so the map in question is injective.
	We induct on $m$ to show that the map is surjective.
	If $m=0$, then $s$ is an isomorphism and $s_0$ is a bijection, hence $\sections(s)$ and $\sections(s_0)$ are both one element sets and surjectivity follows.

	Suppose that $m\geq 1$ and the result is known for $m-1$.
	By choosing a vertex $v$ with $s(v)$ an edge, we may factor $s$ as 
	\[ \begin{tikzcd}
	G_n \rar{s'} \arrow[rr, bend right, "s"] & G_{n-1} \rar{s''} 
		& G_{n-m}
	\end{tikzcd} \]
	with $G_{n-1} = G_n(\uparrow_v)$ and $s', s'' \in \Gamma^-$.

	Let $\alpha_0 \colon \edge(G_{n-m}) \to \edge(G_n)$ be a section of $s_0$.
	Then \[ \Id_{\edge(G_{n-m})} = s_0 \alpha_0 = (s_0'' s_0') \alpha_0 = s_0''(s_0'\alpha_0) \]
	so $s_0'\alpha_0 \in \sections(s_0'')$.
	By the induction hypothesis, we know there is a $\beta \in \sections(s'')$ so that $\beta_0 = s_0'\alpha_0$.

	The codegeneracy map $s'$ identifies exactly two edges $e_1$ and $e_2$ and is the identity elsewhere.
	Write $\bar e = s'_0(e_1) = s'_0(e_2)$ for the common image.
	For $i=1,2$, write $\gamma^i$ for the section of $s'$ (guaranteed by Lemma \ref{codegensection}) with $e_i = \gamma^i_0(\bar e)$.
	We claim that $(\gamma^1\beta)_0 = \alpha_0$ or $(\gamma^2\beta)_0 = \alpha_0$.
	This implies the result, for if $s_0(\gamma^i\beta)_0 = \Id$ then $s(\gamma^i\beta) = \Id$, hence $\sections(s) \to \sections(s_0)$ is surjective.

	We have that $\gamma_0^i s_0'(\tilde e) = \tilde e$ for $\tilde e\neq e_1, e_2$.
	If $e\neq s_0(e_1)$ then $\alpha_0(e) \notin \{e_1, e_2\}$, hence $(\gamma^i \beta)_0 (e) = \gamma^i_0 s_0' \alpha_0 (e) =  \alpha_0(e)$.
	We thus need to consider the case of $e_0 = s_0(e_1) = s_0(e_2)$ and the possible values of $\alpha_0(e_0)$.
	\begin{itemize}
		\item If $\alpha_0(e_0) = e_1$, then 
		\[
		 	(\gamma^1\beta)_0(e_0) = \gamma_0^1s_0'\alpha_0(e_0) = \gamma_0^1 s_0'(e_1) = \gamma_0^1(\bar e) = e_1 = \alpha_0(e_0).
		 \] 
		\item If $\alpha_0(e_0) = e_2$, then
		\[
			(\gamma^2\beta)_0(e_0) = \gamma_0^2s_0'\alpha_0(e_0) = \gamma_0^2 s_0'(e_2) = \gamma_0^2(\bar e) = e_2 = \alpha_0(e_0).
		\]
		\item If $\alpha_0(e_0) = e' \notin \{e_1, e_2\}$, then for $i=1,2$ we have
		\[
			(\gamma^i\beta)_0(e_0) = \gamma_0^i s_0' \alpha_0(e) = \gamma_0^i s_0'(e') = e' = \alpha_0(e).
		\]
	\end{itemize}
	Hence either $(\gamma^1 \beta)_0 = \alpha_0$ or $(\gamma^2\beta)_0 = \alpha_0$, and we have constructed a section hitting $\alpha_0$.
\end{proof}

\begin{definition}
\label{degengraphex}
Suppose $\sK \in \gupcset$ and $x \in \sK(G)$.  Then we say $x$ is \textbf{degenerate} if there exist
a non-empty composition of codegeneracy maps $\sigma \colon \Gamma(G) \to \Gamma(G')$ and 
$y \in \sK(G')$
such that
\begin{equation*}
\label{xsigmay} \tag{$\clubsuit$}
x = \sigma^*y.
\end{equation*}
If no such pair $(\sigma, y)$ exists, then $x$ is said to be \textbf{non-degenerate}.
\end{definition}

The following observation is the graphical version of the Eilenberg-Zilber lemma \cite{gz} (II.3 Proposition).

\begin{lemma}
\label{eilenbergzilber}
Suppose $\sK \in \gupcset$ and $x \in \sK(G)$.  The element $x$ is degenerate if and only if 
there exists a unique pair $(\sigma, y)$ as in \eqref{xsigmay} with $y$ non-degenerate.
\end{lemma}

\begin{proof}
Suppose $x$ is degenerate.  If $s \colon \Gamma(G) \to \Gamma(G')$ is a codegeneracy map, then
\[
|\vertex(G')|=|\vertex(G)| - 1.
\]
Since $|\vertex(G)|$ is finite, a finite induction shows that there exists a pair $(\sigma,y)$ as in \eqref{xsigmay} such that $y \in \sK(G')$ is non-degenerate.  To show that this pair is unique, suppose for $i=1,2$ we have
\[
\sigma_i \colon \Gamma(G) \to \Gamma(G_i), y_i \in \sK(G_i)
\]
satisfying $\sigma_i^*y_i = x$ with $y_i$ nondegenerate.  We first show that $y_1 = y_2$.

By repeatedly using Lemma \ref{codegensection}, the non-empty composition $\sigma_1$ of codegeneracy maps has a section $f$, which is a non-empty composition of coface maps.  Then
\[
\begin{aligned}
y_1 
&= (\sigma_1 f)^*y_1\\
&= f^*\sigma_1^*y_1\\
&= f^*x\\
&= f^*\sigma_2^*y_2\\
&= (\sigma_2 f)^*y_2.
\end{aligned}
\]
To show that $y_1=y_2$, it suffices to show that $\sigma_2f = \Id$.  By \cite[Theorem 6.57]{hry} the map
\[
\Gamma(G_1) \xrightarrow{\sigma_2 f}  \Gamma(G_2)
\]
has a decomposition into codegeneracy maps, followed by an isomorphism, then followed by coface maps.  However, since $y_1$ is non-degenerate, it follows that $\sigma_2 f$ is a composition of an isomorphism followed by  coface maps.  Since a coface map increases the number of vertices by $1$, we have $|\vertex(G_1)| \leq |\vertex(G_2)|$.  

A symmetric argument gives the reverse inequality, hence 
\[
|\vertex(G_1)| = |\vertex(G_2)|.
\]
This implies that $\sigma_2 f$ is an isomorphism.  Since $\sigma_2f$ is a composition of coface maps followed by codegeneracy maps, we have 
\begin{equation}\label{sigma2f}
	\sigma_2f = \Id.
\end{equation}
Thus $G_1 = G_2$ and $y_1 = y_2$.

To see that $\sigma_1$ and $\sigma_2$ are the same map,
it suffices to show that they are equal on $\edge(G)$ because codegeneracy maps are uniquely determined by their actions on edge sets.  
We will show that $\sigma_1$ and $\sigma_2$ agree on an arbitrary edge $e\in \edge(G)$ by using the second part of Lemma \ref{codegensection}.
That is, we know there exists a section $g \colon \Gamma(G_1) \to \Gamma(G)$ of $\sigma_1$ such that
\begin{itemize}
\item
$g$ is a non-empty composition of coface maps, and
\item
$e$ is in the image of $g$, say, $e = g(e')$ for some $e' \in \edge(G_1)$. 
\end{itemize}
Then we have
\[
\sigma_2(e) = \sigma_2 g(e') \overset{\eqref{sigma2f}}{=} e' = \sigma_1 g(e') = \sigma_1(e).
\]
\end{proof}

A commutative square
\[ \begin{tikzcd}
c_1 \rar \dar & c_2 \dar \\
c_3 \rar & c_4
\end{tikzcd} \]
in a category $\mathcal R$ is a \textbf{strong}, or \textbf{absolute}, \textbf{pushout} (see \cite[Definition 5]{bergnerrezk} or \cite[Remark 6.6]{bm}) if its image under the Yoneda embedding 
\[ \begin{tikzcd}
\hom(-,c_1) \rar \dar & \hom(-,c_2) \dar \\
\hom(-,c_3) \rar & \hom(-,c_4)
\end{tikzcd} \]
is a pushout in $\Set^{\mathcal R^{op}}$.

\begin{lemma}\label{strong pushouts}
	If $s_1\colon \Gamma(G) \to \Gamma(H_1)$ and $s_2\colon \Gamma(G) \to \Gamma(H_2)$ are iterated codegeneracy maps, then there exists a strong pushout
	\[ \begin{tikzcd}
	\Gamma(G) \rar{s_1} \dar{s_2} & \Gamma(H_1) \dar{f_1} \\
	\Gamma(H_2) \rar{f_2} & \Gamma(H).
	\end{tikzcd} \]
	Moreover, if $s_1 \neq s_2$, then $f_1$ and $f_2$ are also iterated codegeneracies.
\end{lemma}

\begin{proof}
	If $s_1 = s_2$, take $f_1=f_2 = \Id_{H_1}$, and the above square is a strong pushout. For the remainder of the proof, suppose that $s_1 \neq s_2$.

	The maps $s_i$, $i=1,2$, correspond to graph substitutions
	\[
		H_i = G(\uparrow_{v_i})
	\]
	where $\uparrow_{v_i}$ is inserted at the vertex $v_i$.
	We have $v_1 \neq v_2$ since $s_1 \neq s_2$, and we define
	\[
		H = G(\{ \uparrow_{v_1}, \uparrow_{v_2} \}).
	\]
	By associativity of graph substitution, we have a commutative square
	\[ \begin{tikzcd}
	\Gamma(G) \rar{s_1} \dar{s_2} & \Gamma(H_1) \dar{f_1} \\
	\Gamma(H_2) \rar{f_2} & \Gamma(H).
	\end{tikzcd} \]
	with $f_1$ and $f_2$ are the codegeneracy maps associated to the graph substitutions $H=H_1(\uparrow_{v_2})$ and $H=H_2(\uparrow_{v_1})$, respectively.

	We now must show that the commutative square
	\[ \begin{tikzcd}
	\Gamma[G] \rar{s_1} \dar{s_2} & \Gamma[H_1] \dar{f_1} \\
	\Gamma[H_2] \rar{f_2} & \Gamma[H]
	\end{tikzcd} \]
	is a pushout in $\gupcset$.
	Equivalently, it is enough to show, for any $\sK\in \gupcset$, that
	\[ \begin{tikzcd}
	\sK(H) \rar{f_1^*} \dar{f_2^*} & \sK(H_1) \dar{s_1^*} \\
	\sK(H_2) \rar{s_2^*} & \sK(G)
	\end{tikzcd} \]
	is a pullback in $\Set$.
	We construct a map
	\[
		\Phi\colon \sK(H_1) \underset{\sK(G)}\times \sK(H_2) \to \sK(H)
	\]
	which is inverse to $f_1^* \times f_2^*$.
	Suppose that we have a pair $(x_1, x_2) \in \sK(H_1) \times \sK(H_2)$ with $s_1^*x_1 = s_2^*x_2 \in \sK(G)$.
	This element is degenerate, so by Lemma \ref{eilenbergzilber} there exists a unique pair $(\sigma, y)$ with $\sigma\colon \Gamma(G) \to \Gamma(K)$ a composition of codegeneracy maps, $y$ nondegenerate, and $\sigma^*y=s_i^*x_i$.
	There exists a map $\tilde \sigma\colon \Gamma(H) \to \Gamma(K)$ which is a (possibly empty) composition of codegeneracy maps
	\[ \begin{tikzcd}
	& \Gamma(G) \arrow{dl}[swap]{s_1} \arrow{dr}{s_2} & \\
	\Gamma(H_1) \arrow{dr}[swap]{f_1} & & \Gamma(H_2) \arrow{dl}{f_2} \\
	& \Gamma(H) \dar{\tilde \sigma} &\\
	&\Gamma(K)
	\end{tikzcd} \]
	and $\sigma = \tilde \sigma f_1 s_1 = \tilde \sigma f_2 s_2$.
	Define
	\[
		\Phi(x_1, x_2) = \tilde \sigma^* y \in \sK(H).
	\]
	Then $s_i^* x_i = s_i^* f_i^* \tilde \sigma^* y = s_i^* f_i^* \Phi(x_1, x_2)$ so using any section $d_i$ of $s_i$ (Lemma \ref{codegensection}) we get $x_i = f_i^* \Phi(x_1, x_2)$. 
	Thus $(f_1^* \times f_2^*) \circ \Phi = \Id$.
	On the other hand, suppose that $w\in \sK(H)$.
	If $w = \gamma^* z$ for some nondegenerate $z$ as in Lemma \ref{eilenbergzilber} (where $\gamma$ is a possibly empty composition of codegeneracy maps), then $s_i^*f_i^* w = (f_i s_i)^* \gamma^* z = (\gamma f_i s_i)^* z$, and we see that
	\[
		\Phi(f_1^* w, f_2^* w) = \gamma^* z = w,
	\]
	hence $\Phi\circ (f_1^* \times f_2^*) = \Id$.
\end{proof}

Recall the following from \cite[6.7]{bm}.

\begin{definition}\label{def:EZ cat}
	An \textbf{EZ-category} is a small category, equipped with a degree function from the set of objects to $\mathbb{N}$, such that
	\begin{enumerate}
		\item monomorphisms preserve (resp. raise) the degree if and only if they are invertible (resp. non-invertible);
		\item every morphism factors as a split epimorphism followed by a monomorphism;
		\item any pair of split epimorphisms with common domain has a strong pushout.
	\end{enumerate}
\end{definition}

Any EZ-category $\mathcal R$ is automatically a Reedy category with $\mathcal R^+$ the subcategory of monomorphisms and $\mathcal R^-$ the subcategory of split epimorphisms. This fact will be used in the proof of Theorem \ref{gammaw not ez} to show that $\Gammaw$ cannot be an EZ-category.

The proof of the following lemma rests on a simple fact. In any category, if we have a composition $h = g\circ f$, then
\begin{itemize}
 	\item if $f$ is not a monomorphism then neither is $h$, and
 	\item if $g$ does not admit a section, then neither does $h$.
 \end{itemize} 

\begin{lemma}\label{reverse inclusions}
The set of split epimorphisms is contained in $\Gamma^-$ and the set of monomorphisms is contained in $\Gamma^+$.
\end{lemma}
\begin{proof}
	First observe two facts.
	\begin{itemize}
		\item Codegeneracies are not monomorphisms. A codegeneracy $s$ admits two distinct sections $d^1$ and $d^{-1}$ (corresponding to $b_1$ and $b_{-1}$) according to Lemma \ref{codegensection}; but then $sd^1 = \Id = sd^{-1}$,  
		hence $s$ is not a monomorphism.
		\item Coface maps do not admit sections. This is clear if $f$ is not surjective on edges (in particular, if $f$ is an inner coface map).
		Using the notation of \cite[Definition 6.8]{hry}, 
		if $f$ is a outer coface map, then a hypothetical section $\beta$ would be forced on $H_w$, but has nowhere to send $u$.
	\end{itemize}
	Let $f\colon G \to K$ be a map, and factor as $f = \partial i \sigma$ as in \cite[Lemma 6.60]{hry} with $\partial$ a composition of coface maps, $i$ an isomorphism, and $\sigma$ a composition of codegeneracy maps.
	By the above, if $f$ is a monomorphism, then $\sigma = \Id$, hence $f\in \Gamma^+$. Similarly, if $f$ admits a section, then $\partial = \Id$, hence $f\in \Gamma^-$.
\end{proof}

\begin{lemma}\label{mono edge}
	If $f \colon G \to K$ is an injection on edge sets, then $f$ is a monomorphism.
\end{lemma}
\begin{proof}
	Let $H$ be an arbitrary object of $\Gamma$.
	By \cite[Corollary 6.62]{hry} the vertical maps in the commutative diagram are injections
	\[ \begin{tikzcd}
	\Gamma(H,G) \rar \dar[hook] & \Gamma(H,K) \dar[hook] \\
	\Set(\edge(H), \edge(G)) \rar[hook] & \Set(\edge(H), \edge(K))
	\end{tikzcd} \]
	while the bottom map is an injection by assumption.
	Thus the top map is an injection as well. Since $H$ was arbitrary, $f$ is a monomorphism.
\end{proof}

\begin{theorem}\label{EZ2 theorem}
	The graphical category $\Gamma$ is an EZ-category with \[ \deg \Gamma(G) = |\vertex(G)|. \]
\end{theorem}

\begin{proof}
Since coface maps are injections on edge sets, $\Gamma^+$ is contained in the set of monomorphisms by Lemma \ref{mono edge}, while $\Gamma^-$ is contained in the set of split epimorphisms by Lemma \ref{codegensection}. Thus the usual Reedy factorization of maps implies that (2) holds.

By Lemma \ref{reverse inclusions} we know that $\Gamma^+$ is exactly the set of monomorphisms and $\Gamma^-$ is exactly the set of split epimorphisms.
The first statement and condition (i) of a Reedy category imply that (1) holds. 
The second statement and repeated application of Lemma \ref{strong pushouts} implies that (3) holds.
\end{proof}

The Reedy structure on $\Gammaw$ from Section \ref{sec:reedygammaw} is not of Eilenberg-Zilber type, for $\Gammaw^+$ is not contained in the set of monomorphisms. Indeed, monomorphisms are in particular monomorphisms on edge sets (since if $f \colon G \to H$ is a monomorphism, then so is $\edge(G) = \hom({\uparrow}, G) \to \hom({\uparrow}, H) = \edge(H)$), hence outer contracting coface maps are not monomorphisms. In fact, more is true, and no choice of degree function on $\Gammaw$ will turn it into an EZ-category. Any such category is a generalized Reedy category whose direct part is the class of monomorphisms and whose inverse part is the class of split epimorphisms, so we simply need to exhibit one map which does not factor as a split epimorphism followed by a monomorphism. As usual, by $\Gammaw$ we mean one of the categories $\oldB$ or $\oldC$.

\begin{figure}
\begin{center}
   \includegraphics[width=.5\textwidth]{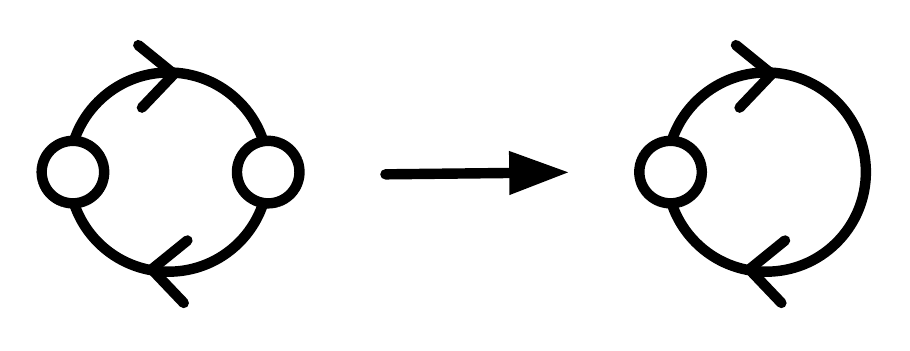}
\end{center}
	\caption{A codegeneracy map in $\Gammaw$}\label{gwnotez_fig}
\end{figure}



\begin{theorem}\label{gammaw not ez}
The category $\Gammaw$ does not admit the structure of an EZ-category.
\end{theorem}
\begin{proof}
We consider the codegeneracy from Figure \ref{gwnotez_fig};
let $G = $ \begin{tikzpicture}
	\draw[line width=1pt] (0,0) circle (.9ex);
	\fill (.9ex,0) circle (.4ex);
	\fill (-.9ex,0) circle (.4ex);
	\draw[line width=.5pt] (-.2ex, 1.15ex) -- (.1ex, .9ex) -- (-.2ex, .65ex);
	\draw[line width=.5pt] (.2ex, -1.15ex) -- (-.1ex, -.9ex) -- (.2ex, -.65ex);
\end{tikzpicture} and $H = $  \begin{tikzpicture}
	\draw[line width=1pt] (0,0) circle (.9ex);
	\fill (-.9ex,0) circle (.4ex);
\end{tikzpicture} $ = \xi^1_1 C_{(1;1)}$. There are two codegeneracy maps $G \to H$, which, as we mentioned in Remark \ref{failure_for_gw}, do not admit sections.
In fact, we will show that there is no factorization of a codegeneracy $s\colon G \to H$ into a split epimorphism followed by a monomorphism.
Thus, the category $\Gammaw$ does not admit any Eilenberg-Zilber structure.

Suppose we have any factorization $G \xrightarrow{g} K \xrightarrow{f} H$ with $f$ a monomorphism and $s = f\circ g$.
As noted in the paragraph preceding this theorem, monomorphisms are in particular injective on edges, so $|\edge(K)| \leq 1$, that is,
\[
	K \in \left\{ \bullet, {\uparrow}, {\wheel}, {\tikz[ baseline=-.7ex] \filldraw[line width=1pt, baseline=-10pt] (0,.7ex) -- (0,-.7ex) circle (2pt);} , { \tikz[ baseline=-.7ex] \filldraw[line width=1pt] (0,-.7ex) -- (0,.7ex) circle (2pt);} , \tikz[ baseline=-.7ex] \filldraw[line width=1pt] (0,-.7ex) circle (2pt) -- (0,.7ex) circle (2pt);, H  \right\}. 
\]
If $K$ is in this list, then $\hom(G,K) = \varnothing$ unless $K={\wheel}$ or $K= H$.
Further, $\hom({\wheel}, H) = \varnothing$ and $\hom(H,H) = \{ \Id_H \}$, so we must have had $g = s$ and $f = \Id_H$.
Since $s$ does not admit a section, the given factorization does not have $g$ a split epimorphism.
\end{proof}

We end this section by showing that $\Gamma$ satisfies \cite[Definition 5.1]{minimalfib}, which is a different notion of Eilenberg-Zilber category.
These are generalized Reedy categories whose inverse category is the subcategory of split epimorphisms, such that if two split epimorphisms have the same set of sections then they are equal. Since we already know that $\Gamma^-$ is the set of split epimorphisms, we need only show the following.

\begin{proposition}\label{EZ1 prop}
	If $f, f' \in \Gamma^-$ have the same set of sections, then $f=f'$.
\end{proposition}

\begin{proof}
If $p$ and $q$ are two surjections of sets with $\sections(p) = \sections(q)$, then $p=q$.

The hypothesis of the proposition says that $\sections(f) = \sections(f')$, which implies that $\sections(f_0) = \sections(f_0')$ by Proposition \ref{sections_bijection}. 
Since $f_0$ and $f_0'$ are surjections of sets which share a common set of sections, $f_0 = f_0'$. By \cite[Corollary 6.62]{hry}, $f=f'$.
\end{proof}

\begin{corollary}\label{EZ1 COROLLARY}
	The graphical category $\Gamma$ is an Eilenberg-Zilber category in the sense of \cite[5.1]{minimalfib}. In other words, $\Gamma$ is a generalized Reedy category so that
	\begin{itemize}
		\item $\Gamma^-$ is the subcategory of split epimorphisms, and
		\item if $s, s'\in \Gamma^-$ have the same set of sections, then $s=s'$.
	\end{itemize}
	\qed
\end{corollary}

\section{Segal properads}\label{sec:weaksegalproperads}

Let $\sSet$ be the category of simplicial sets.
If $\mathcal R$ is any generalized Reedy category, there is a model structure on $\sSet^{\mathcal R}$ as proved in \cite{bm}.
We will say that an object is Reedy fibrant if it is fibrant in this model structure.

\begin{definition}\label{weak segal properad}
	An object $\sK \in \gupcsset$ (resp. $\in \gwheelcsset$) is called a Segal (wheeled) properad if the following hold:
	\begin{itemize}
		\item $\sK$ is Reedy fibrant,
	 	\item $\sK(\uparrow) = \Delta^0$, and
	 	\item for all\footnote{Technically, if we are working in $\Gamma$ we do not yet have Segal maps associated to $G = {\uparrow}$ or $G= C_{(n;m)}$ -- this is not particularly relevant, as the Segal core inclusions should just be isomorphisms in this case, so the Segal map is even an isomorphism.} graphs $G$, the Segal map
	 	\[
	 		\epsilon_G \colon \sK(G) \to \map(\Sc[G], \sK)
	 	\]
	 	is a weak equivalence of simplicial sets.
	\end{itemize} 
\end{definition}

The purpose of this section is to point out that there is a model category whose fibrant objects are precisely the Segal (wheeled) properads.
These should be thought of as one-colored properads so that properadic composition is only weakly defined.

Let $\gupcsset_*$ (resp. $\gwheelcsset_*$) be the full subcategory consisting of presheaves $\sK$ with $\sK(\uparrow) = \Delta^0$.
This is a reflective subcategory, where the left adjoint to the inclusion is `reduction' $\sK \mapsto \sK_*$ (see analogues in the dendroidal setting in \cite[proof of 9.4]{cm-simpop} or \cite[Definition 3.6]{bh1}, with the general case in \cite{hco}).

\begin{proposition}\label{reedy star model}
	The categories $\gupcsset_*$ and $\gwheelcsset_*$ admit left proper, cellular, simplicial model category structures, lifted from the Berger-Moerdijk-Reedy model structures on $\gupcsset$ and $\gwheelcsset$.
\end{proposition}
\begin{proof}
Let $\mathcal R$ be either $\Gamma$ or $\Gammaw$.
Then $\mathcal R$ contains a unique object of degree 0, namely $\uparrow$.
Further, if $G$ is a graph, then there are only finitely many maps in $\mathcal R^+$ with codomain $G$. Thus this proposition is a special case of the relevant theorems from \cite{hco} (Theorem 7.13 and Proposition 7.16).
\end{proof}

\begin{theorem}\label{reduced ms}
	The category $\gupcsset_*$ (resp. $\gwheelcsset_*$) admits a model structure with fibrant objects the Segal properads (resp. the Segal wheeled properads).
\end{theorem}
\begin{proof}
For notational reasons we prove the theorem for $\gupcsset_*$, with the understanding that the proof for $\gwheelcsset_*$ is nearly identical.

Notice that an object $\sK \in \gupcsset_*$ is fibrant if and only if it is fibrant in the larger category $\gupcsset$ (by Corollary 7.15 in \cite{hco}).

We apply the composite $\gupcset \to \gupcsset \to \gupcsset_*$ to each Segal core inclusion $\Sc[G] \to \Gamma[G]$, and call the resulting set 
\[ \{ \Sc[G]_* \to \Gamma[G]_* \} \] the \emph{reduced Segal core inclusions}.
Domains and codomains of these maps are cofibrant since $\Sc[G]$ and $\Gamma[G]$ are cofibrant and the reduction functor is left Quillen.

Since the model structure from Proposition \ref{reedy star model} is left proper and cellular, we may apply left Bousfield localization at the set of reduced Segal core inclusions \cite[4.1.1]{hirschhorn}.
The fibrant objects in the localized model structure are those objects $\sK$ which are fibrant in the unlocalized model structure (which is the same as being Reedy fibrant in $\gupcsset$), such that
\[
	\map^h(\Sc[G]_*, \sK) \leftarrow \map^h(\Gamma[G]_*, \sK)
\]
is a weak equivalence of simplicial sets for all $G$ (where $\map^h$ is the homotopy function complex).
In any simplicial model category with mapping spaces $\map$, if $A$ is cofibrant and $Z$ is fibrant, then $\map^h(A,Z) \simeq \map(A,Z)$.
Since $\Sc[G]_*$ and $\Gamma[G]_*$ are both cofibrant and $\sK$ is fibrant, we are thus asking that
\[
	 \map(\Sc[G], \sK) = \map(\Sc[G]_*, \sK) \leftarrow \map(\Gamma[G]_*, \sK) =  \sK(G)
\]
be a weak equivalence of simplicial sets. This weak equivalence is precisely the third requirement from Definition \ref{weak segal properad}. Since every object in the category satisfies $\sK(\uparrow) = \Delta^0$, the fibrant objects in the localized model structure are precisely the Segal properads.
\end{proof}

The category of one-colored simplicial (wheeled) properads also admits a model structure (lifted from that on $\sSet^{\Sigma^{op}\times \Sigma}$ by, say, \cite[2.1]{bmresolution} using the operad from \cite[14.1.2, p.~265]{yj}). In light of \cite[1.1]{bh1}, we conjecture that this model structure is Quillen equivalent to the one from Theorem \ref{reduced ms}.


\begin{thebibliography}{10}

\bibitem{bmresolution}
Clemens Berger and Ieke Moerdijk, \emph{Resolution of coloured operads and
  rectification of homotopy algebras}, Categories in algebra, geometry and
  mathematical physics, Contemp. Math., vol. 431, Amer. Math. Soc., Providence,
  RI, 2007, pp.~31--58. 

\bibitem{bm}
\bysame, \emph{On an extension of the notion of {R}eedy category}, Math. Z.
  \textbf{269} (2011), no.~3-4, 977--1004. 

\bibitem{bergner}
Julia~E. Bergner, \emph{A model category structure on the category of
  simplicial categories}, Trans. Amer. Math. Soc. \textbf{359} (2007), no.~5,
  2043--2058. 

\bibitem{juliesurvey}
\bysame, \emph{A survey of {$(\infty,1)$}-categories}, Towards higher
  categories, IMA Vol. Math. Appl., vol. 152, Springer, New York, 2010,
  pp.~69--83. 

\bibitem{bh1}
Julia~E. Bergner and Philip Hackney, \emph{Group actions on {S}egal operads},
  Israel J. Math. \textbf{202} (2014), no.~1, 423--460. 

\bibitem{bergnerrezk}
Julia~E. Bergner and Charles Rezk, \emph{Reedy categories and the
  {$\varTheta$}-construction}, Math. Z. \textbf{274} (2013), no.~1-2, 499--514.
  

\bibitem{bv}
J.~M. Boardman and R.~M. Vogt, \emph{Homotopy {I}nvariant {A}lgebraic
  {S}tructures on {T}opological {S}paces}, Lecture Notes in Mathematics, Vol.
  347, Springer-Verlag, Berlin, 1973. 

\bibitem{MR2294028}
Denis-Charles Cisinski, \emph{Les pr\'efaisceaux comme mod\`eles des types
  d'homotopie}, Ast\'erisque (2006), no.~308, xxiv+390. 

\bibitem{cm-ho}
Denis-Charles Cisinski and Ieke Moerdijk, \emph{Dendroidal sets as models for
  homotopy operads}, J. Topol. \textbf{4} (2011), no.~2, 257--299. 

\bibitem{cm-ds}
\bysame, \emph{Dendroidal {S}egal spaces and $\infty$-operads}, J. Topol.
  \textbf{6} (2013), no.~3, 675--704.

\bibitem{cm-simpop}
\bysame, \emph{Dendroidal sets and simplicial operads}, J. Topol. \textbf{6}
  (2013), no.~3, 705--756.

\bibitem{gz}
P.~Gabriel and M.~Zisman, \emph{Calculus of {F}ractions and {H}omotopy {T}heory},
  Ergebnisse der Mathematik und ihrer Grenzgebiete, Band 35, Springer-Verlag
  New York, Inc., New York, 1967. 

\bibitem{hry}
Philip Hackney, Marcy Robertson, and Donald Yau, \emph{Infinity {P}roperads and
  {I}nfinity {W}heeled {P}roperads}, Lecture Notes in Mathematics, vol. 2147,
  Springer, Cham, 2015. 

\bibitem{hco}
\bysame, \emph{Higher cyclic operads}, 2016, preprint (version 2),
  \href{https://arxiv.org/abs/1611.02591v2}{arXiv:1611.02591v2}.

\bibitem{hry15}
\bysame, \emph{A simplicial model for infinity properads}, Higher Structures
  \textbf{1} (2017), no.~1, 1--21.

\bibitem{hirschhorn}
Philip~S. Hirschhorn, \emph{Model {C}ategories and {T}heir {L}ocalizations},
  Mathematical Surveys and Monographs, vol.~99, American Mathematical Society,
  Providence, RI, 2003. 

\bibitem{hv15}
Philip~S. Hirschhorn and Ismar Voli\'c, \emph{Functors between {R}eedy model
  categories of diagrams}, 2015, preprint,
  \href{https://arxiv.org/abs/1511.04809}{arXiv:1511.04809}.

\bibitem{joyaltheoryqcat}
Andr\'e Joyal, \emph{The {T}heory of {Q}uasi-{C}ategories and its
  {A}pplications}, 2008, Available at
  \url{http://mat.uab.cat/~kock/crm/hocat/advanced-course/Quadern45-2.pdf}.

\bibitem{htt}
Jacob Lurie, \emph{Higher {T}opos {T}heory}, Annals of Mathematics Studies,
  vol. 170, Princeton University Press, Princeton, NJ, 2009. 

\bibitem{MR2483835}
M.~Markl, S.~Merkulov, and S.~Shadrin, \emph{Wheeled {PROP}s, graph complexes
  and the master equation}, J. Pure Appl. Algebra \textbf{213} (2009), no.~4,
  496--535. 

\bibitem{minimalfib}
Ieke Moerdijk and Joost Nuiten, \emph{Minimal fibrations of dendroidal sets},
  Algebr. Geom. Topol. \textbf{16} (2016), no.~6, 3581--3614. 

\bibitem{mw}
Ieke Moerdijk and Ittay Weiss, \emph{Dendroidal sets}, Algebr. Geom. Topol.
  \textbf{7} (2007), 1441--1470. 

\bibitem{vallette}
Bruno Vallette, \emph{A {K}oszul duality for {PROP}s}, Trans. Amer. Math. Soc.
  \textbf{359} (2007), no.~10, 4865--4943. 

\bibitem{yau_wheel}
Donald Yau, \emph{Dwyer-{Kan} homotopy theory of algebras over operadic
  collections}, 2016, preprint,
  \href{http://arxiv.org/abs/1608.01867}{arXiv:1608.01867} [math.AT].

\bibitem{yj}
Donald Yau and Mark~W. Johnson, \emph{A {F}oundation for {P}{R}{O}{P}s,
  {A}lgebras, and {M}odules}, Math. Surveys and Monographs, Vol. 203, Amer.
  Math. Soc., Providence, RI, 2015.

\end{thebibliography}

\providecommand{\bysame}{\leavevmode\hbox to3em{\hrulefill}\thinspace}

\end{document}